\documentclass{amsart}

\usepackage{amssymb}

\newcommand{\RR}{{\mathbb R}}
\newcommand{\NN}{{\mathbb N}}

\def\A{{\mathcal A}}  
\def\M{{\mathcal M}}
\def\S{{\mathcal S}}

\def\gap{\hbox{\rm gap}}
\def\rep{\hbox{\rm rep}}

\numberwithin{equation}{section}

\newtheorem{theo}{Theorem}    
\newtheorem{prop}[theo]{Proposition}  
\newtheorem{coro}[theo]{Corollary}  
\newtheorem{lemma}[theo]{Lemma} 
\newtheorem{conj}[theo]{Conjecture} 
\theoremstyle{definition}
\newtheorem{defi}[theo]{Definition}  
\newtheorem{rema}[theo]{Remark}  
\newtheorem{exam}[theo]{Example}  

\begin{document}

\title[Syndeticity and independent substitutions]{Syndeticity and independent substitutions}
\author{Fabien Durand}
\address[F.D.]{\newline
Universit\'e de Picardie Jules Verne\newline
Laboratoire Ami\'enois de Math\'ematiques Fondamentales et Appliqu\'ees\newline
CNRS-UMR 6140\newline
33 rue Saint Leu\newline
80039 Amiens Cedex \newline
France.}
\email{fabien.durand@u-picardie.fr}
\author{Michel Rigo}
\thanks{This work has been supported by a grant ``Accord de coop\'eration CNRS/CGRI-FNRS, 
Commissariat G\'en\'eral des Relations Internationales de la Communaut\'e Fran\c{c}aise''}
\address[M.R.]{\newline
Universit\'e de Li\`ege\newline
Institut de Math\'ematique\newline
Grande traverse 12 (B 37)\newline
B-4000 Li\`ege\newline
Belgium.}
\email{M.Rigo@ulg.ac.be}

\begin{abstract}
  We associate in a canonical way a substitution to any abstract
  numeration system built on a regular language. In relationship with
  the growth order of the letters, we define the notion of two
  independent substitutions. Our main result is the following. If a
  sequence $x$ is generated by two independent substitutions, at
  least one being of exponential growth, then the factors of $x$
  appearing infinitely often in $x$ appear with bounded gaps. As an
  application, we derive an analogue of Cobham's theorem for two
  independent substitutions (or abstract numeration systems) one with
  polynomial growth, the other being exponential.
\end{abstract}

\maketitle

\section{Introduction}

A set $E\subset \NN$ is {\it $p$-recognizable} for some
$p\in\NN\setminus\{0,1\}$, if the language consisting of the $p$-ary
expansions of the elements in $E$ is recognizable by a finite
automaton \cite{Ei}. In 1969, A.~Cobham obtained the following result
\cite{Co1}. {\it Let $p,q\ge 2$ be two multiplicatively independent
  integers (i.e., $p^k \not = q^\ell$ for all integers $k,\ell>0$). A
  set $E\subset\NN$ is both $p$-recognizable and $q$-recognizable if
  and only if $E$ is a finite union of arithmetic progressions.}
 
A key part in all known proofs of this seminal theorem (and this
remark stands also for generalizations to non-standard positional
numeration systems) is to show that $E$ is {\it syndetic} (i.e., the
difference between two consecutive elements of $E$ is bounded), see
\cite{Ha,Du1,Du2}.

In this paper we study this syndeticity problem for a larger class of
numeration systems namely, for numeration systems built on infinite
regular languages, the so-called {\it abstract numeration systems}
\cite{LR}. In particular, these systems contain classical numeration
systems like the $k$-ary system or the Fibonacci system, but also more
``exotic'' systems for which the language of the numeration contains a
number of words of length $n$ bounded by a polynomial in $n$ (which is
contrasting with the usual exponential paradigm).

In 1972, A.~Cobham characterized $p$-recognizable sets of integers in
terms of constant length substitutions.  It turns out to be mainly the
same for abstract numeration systems (this is the purpose of Section
\ref{sec3}).  Hence we will often say that a set of integers
recognizable with respect to some abstract numeration system is {\it
  generated by a substitution}.  This will enable us to solve the
syndeticity problem for abstract numeration systems in terms of
substitutions. Let us also observe that with the formalism of
substitutions and in connection with the constructions of Section
\ref{sec3}, Cobham's theorem obtained in \cite{Du1} can be
directly translated for a large class of abstract numeration systems
(namely, those giving rise to substitutions of exponential growth
satisfying the assumptions of \cite{Du1}).

In \cite{Co1,Co2,Du1,Du2}, the involved substitutions $\sigma$ are
(exponentially) growing, meaning that the length of $\sigma^n(a)$ goes
to infinity with $n$, for all letters $a$. (This implies in particular
that one of the letter is of exponential growth and that none of them
has polynomial growth.)  The substitutions corresponding to abstract
numeration systems do not have this latter property: they can be
non-growing (in the polynomial case) and even worse, erasing. We take
care of this extra difficulty in Section \ref{sec4}.

The notion of multiplicatively independent integers can be generalized
to these substitutions by considering the maximal growth rate of the
letters and we are thus able to define ``independent'' substitutions.
Our main result (Theorem \ref{the:boundedgap}) can be roughly stated
as follows:

{\it If a set of integers $E$ is generated by two independent
  substitutions (one having exponential growth), then $E$ is
  syndetic}. We are not able to give a complete proof in the case of
two independent substitutions both having polynomial growth.

To conclude this paper, we obtain easily from the syndeticity an
analogue of Cobham's theorem for two substitutions (or equivalently
for two abstract numeration systems): one of exponential growth and
the other one of polynomial growth. Combined with the main result of
\cite{Du1}, an extended version of Cobham's theorem follows.

\section{Words, morphisms, substitutions and numeration systems}

The aim of this section is just to recall classical definitions and
notation.

\subsection{Words and sequences} 
An {\it alphabet} $A$ is a finite set of elements called {\it
  letters}. A {\it word} over $A$ is an element of the free monoid
generated by $A$, denoted by $A^*$. Let $x = x_0x_1 \cdots x_{n-1}$
(with $x_i\in A$, $0\leq i\leq n-1$) be a word, its {\it length} is
$n$ and is denoted by $|x|$. The number of occurrences of a letter
$a\in A$ in the word $w$ is denoted $|w|_a$ and if $E$ is a subset of
$A$, then $|w|_E$ is a shorthand for $\sum_{e\in E}|w|_e$. The {\it
  empty word} is denoted by $\epsilon$, $|\epsilon| = 0$. The set of
non-empty words over $A$ is denoted by $A^+$. The elements of
$A^{\NN}$ are called {\it sequences}. If $x=x_0x_1\cdots$ is a
sequence (with $x_i\in A$, $i\in \NN$) and $I=[k,l]$ an interval of
$\NN$ we set $x_I = x_k x_{k+1}\cdots x_{l}$ and we say that $x_{I}$
is a {\it factor} of $x$.  If $k = 0$, we say that $x_{I}$ is a {\it
  prefix} of $x$. The set of factors of length $n$ of $x$ is written
$L_n(x)$ and the set of factors of $x$, or the {\it language} of $x$,
is noted $L(x)$. The {\it occurrences} in $x$ of a word $u$ are the
integers $i$ such that $x_{[i,i + |u| - 1]}= u$. When $x$ is a word,
we use the same terminology with similar definitions.

The sequence $x$ is {\it ultimately periodic} if there exist a word
$u$ and a non-empty word $v$ such that $x=uv^{\omega}$, where
$v^{\omega}= vvv\cdots $. Otherwise we say that $x$ is {\it
  non-periodic}. It is {\it periodic} if $u$ is the empty word. A
sequence $x$ is {\it uniformly recurrent} if every factor of $x$
appears infinitely often in $x$ and for each factor $u$ the greatest
difference of two successive occurrences of $u$ is bounded.

\subsection{Morphisms and matrices} 
Let $A$ and $B$ be two alphabets. A {\it morphism} $\tau$ is a map
from $A$ to $B^*$. Such a map induces by concatenation a morphism from
$A^*$ to $B^*$. If $\tau (A)$ is included in $B^+$, it induces a map
from $A^{\NN}$ to $B^{\NN}$. These two maps are also called $\tau$.
With the morphism $\tau$ is naturally associated the matrix $M_{\tau} =
(m_{i,j})_{i\in B , j \in A }$ where $m_{i,j}$ is the number of
occurrences of $i$ in the word $\tau(j)$.

Let $M$ be a square matrix, we call {\it dominant eigenvalue of $M$}
an eigenvalue $r$ such that the modulus of all the other eigenvalues
do not exceed the modulus of $r$. A square matrix is called {\it
  primitive} if it has a power with positive coefficients. In this
case the dominant eigenvalue is unique, positive and it is a simple
root of the characteristic polynomial. This is Perron-Frobenius Theorem (see
for instance \cite{LindMarcus}).

\subsection{Substitutions and substitutive sequences}

A {\it substitution} is a morphism $\tau : A \rightarrow
A^{*}$. 
In all this paper, and without exception, a substitution $\tau$ is assumed 
to fulfill the following hypothesis : 
There exists a letter $a\in A$ with

\begin{enumerate}
\item
$\lim_{n\rightarrow +\infty} |\tau^n (a)| = +\infty$ and 
\item
$\tau (a) = au$ for some $u\in A^*$.  
\end{enumerate}

Whenever the matrix associated to $\tau $ is
primitive we say that $\tau$ is a {\it primitive substitution}.  We
say $\tau$ is a {\it growing} substitution if $\lim_{n\rightarrow
  +\infty} |\tau^n (b)| = +\infty$ for all $b\in A$. We say $\tau$ is
{\it erasing} if there exists $b\in A$ such that $\tau (b)$ is the
empty word.  A fixed point of $\tau $ is a sequence $x=(x_n ; n\in \NN
)$ such that $\tau (x) = x$. We say it is a {\it proper fixed point}
if all letters of $A$ have an occurrence in $x$. We observe that all
proper fixed points of $\tau$ have the same language. Notice that each
substitution has at least one proper fixed point.  Let $x$ be a proper
fixed point of $\tau$. We define
$$
L(\tau )
=
\left\{
x_{[i,j]} ; i,j\in \NN , i\leq j 
\right\} .
$$

\begin{exam}
The substitution $\tau $ defined by $\tau (a)=aaab$,
$\tau (b) = bc$ and $\tau (c) = b$ has two fixed points, one is
starting with the letter $a$ and is proper and the other one is
starting with the letter $b$ and is not proper.
\end{exam}

Let $B$ be another alphabet and $y\in B^\NN$. Let $\S$ be a set of
substitutions. We say that $y$ is {\it substitutive in} $\S$ if $y =
\phi (x)$ where $x\in A^\NN$ is a proper fixed point of $\tau \in \S$
and $\phi : A\to B^*$ is a {\it letter-to-letter morphism}, i.e.,
$\phi (A)$ is a subset of $B$.

\subsection{Automata}
We assume that the reader has some basic knowledge in automata theory,
see for instance \cite{Ei}.  A {\it deterministic finite automaton}
over $A$ or simply a DFA is a $5$-tuple
$\mathcal{M}=(Q,q_0,F,A,\delta)$ where $Q$ is the finite set of
states, $q_0\in Q$ is the initial state, $F\subseteq Q$ is the set of
final states and $\delta:Q\times A\to A$ is the (partial) transition
function. A DFA is {\it complete} if $\delta$ is a total function. As
usual, $\delta$ can be naturally extended to $Q\times A^*$. With the DFA
$\mathcal{M}$ is associated the matrix
$M_{\mathcal{M}}=(m_{i,j})_{i,j\in Q}$ where $m_{i,j}=\#\{a\in A;
\delta(j,q)=i\}$.

If $L$ is a regular language then the trim minimal automaton of $L$ is
said to be the {\it canonical automaton} of $L$. Recall that an
automaton is {\it trim} (or reduced) if it accessible and
coaccessible, i.e., every state is reachable from $q_0$ and every
state reaches a final state. 
Let
$\mathcal{M}=(Q,q_0,F,A,\delta)$ be a DFA and $L\subseteq A^*$ be a
regular language with $\mathcal{A}=(Q',q_0',F',A,\delta')$ as
canonical automaton. Then $\mathcal{M}$ is said to be an {\it
  $L$-automaton} if there exists an onto mapping $\Phi:Q\to Q'$ such
that
\begin{enumerate}
 \item $\Phi(q_0)=q_0'$,
 \item $\Phi(F)\subseteq F'$,
 \item $\forall q\in Q$, $\forall a\in A$:
 $\Phi(\delta(q,a))=\delta'(\Phi(q),a)$.
\end{enumerate}
In the latter condition, if $\delta(q,a)$ is not defined then
$\delta'(\Phi(q),a)$ is not defined, and conversely.
Notice that this kind of definition can also be found in \cite{BH}
where linear numeration systems related to a Pisot number are
investigated.

\begin{rema} Changing the set of final states in an $L$-automaton
  allows this automaton to recognize exactly the language $L$. With
  the same notation as before, it suffices to take $\Phi^{-1}(F')$ as
  set of final states for the $L$-automaton.
\end{rema}

\subsection{Abstract numeration systems}

If the alphabet $A$ is totally ordered then we can enumerate the words of
$A^*$ by the {\it genealogical ordering} defined as follows. Let $x,y$
be two words over $A$, we say $x<y$ if $|x|<|y|$ or if $|x|=|y|$ and
there exist $a,b\in A$, $u,x',y'\in A^*$ such that $a<b$, $x=uax'$ and
$y=uby'$. Enumerating the words of an infinite regular language $L$
over a totally ordered alphabet $(A,<)$ by increasing genealogical
order gives a one-to-one correspondence between $\NN$ and $L$ (see
\cite{LR}).  We say that the $(n+1)$th word $w$ in the genealogically
ordered language $L$ is the representation of $n$ in the {\it abstract
  numeration system} $S=(L,A,<)$ and we write $\rep_S(n)=w$. In
particular, if $E$ is a subset of $\NN$ then $\rep_S(E)$ is a subset
of $L$. We say that $E$ is {\it $S$-recognizable} if $\rep_S(E)$ is a
regular language.  The {\it characteristic sequence} of $E$ is the
sequence $\chi_E=x_0x_1\cdots \in \{0,1\}^\NN$ such that $x_i=1$ if
and only if $i$ belongs to $E$.

\begin{exam}
  Let $A=\{0,\ldots,k-1\}$ for some $k\ge 2$. The language
  $$L=\{\epsilon\}\cup\{1,\ldots,k-1\}\{0,\ldots,k-1\}^*$$ genealogically
  ordered with the usual ordering of the digits gives the classical
  $k$-ary system.  Let $B=\{0,1\}$. Enumerating the words of
  $M=\{\epsilon\}\cup 1 \{0,01\}^*$ gives exactly the Fibonacci
  system. These two examples are special cases of linear numeration
  systems whose characteristic polynomial is the minimal polynomial of
  a Pisot number (in this setting, it is well known that the language
  of the numeration is regular \cite{BH}). All the systems of this
  kind are therefore special cases of abstract numeration systems.
\end{exam}

\begin{exam}\label{exaab}
  Let us consider an abstract numeration system which is no more
  positional (i.e., not built on a strictly increasing sequence of
  integers). Let $A=\{a,b\}$ with $a<b$.  The first words of
  $L=a^*b^*$ enumerated by genealogical order are
$$\epsilon,a,b,aa,ab,bb,aaa,aab,abb,bbb,aaaa\ldots$$
For instance, $\rep_S(5)=bb$ and
$\rep_S^{-1}(a^*)=\{0,1,3,6,10,\ldots\}=E_a$ is an $S$-recognizable
subset of $\NN$ (formed of triangular numbers). For such a system,
$\rep_S^{-1}(a^pb^q)=\frac12(p+q)(p+q+1)+q$ and we cannot mimic
positional systems where one can define ``weight'' to the ``digits''
$a$ and $b$. Moreover we can already notice that $\#(L\cap A^n)=n+1$
has a polynomial behavior (contrasting with systems built on Pisot
numbers which always have an exponential behavior).
\end{exam}

\section{The link between substitutions and numeration systems.}\label{sec3}

In this section, we associate a substitution $\sigma$ to any
$S$-recognizable set $E$ of integers for a given abstract numeration
system $S$. One of the fixed point $z$ of $\sigma$ is such that
$f(z)=\chi_E$ for some (possibly erasing) morphism $f$.

\begin{lemma}\label{lemma:Srec} 
  Let $S=(L,A,<)$ be a numeration system.  A set $E\subset\mathbb{N}$
  is $S$-recognizable if and only if $\rep_S(E)$ is accepted by an
  $L$-automaton.
\end{lemma}

\begin{proof} Assume that $E$ is $S$-recognizable. So there exists a
  complete and accessible DFA $\mathcal{M}=(Q,q_0,F,A,\delta)$
  accepting exactly $\rep_S(E)$. We denote by
  $\mathcal{A}=(Q',q_0',F',A,\delta')$ the canonical automaton of $L$.
  Consider the ``product'' automaton
 $$\mathcal{P}=(Q'\times Q, (q_0',q_0),
 F'\times F, A, \mu)$$ where the transition function
 $\mu$ is defined, for all $(q,r)\in Q'\times Q$ and all
 $a\in A$ such that $\delta'(q,a)$ exists, by
 $$\mu((q,r),a)=(\delta'(q,a),\delta(r,a)).$$
 Clearly, $\mathcal{P}$
 is an $L$-automaton accepting $\rep_S(E)$. It suffices to consider the
 application $\Phi:Q'\times Q\to Q'$ mapping $(q,r)$ onto $q$.
\end{proof}

\begin{defi}\label{defisigmam}
Let $A=\{a_1<\cdots <a_k\}$ be a totally ordered alphabet. To any DFA
$\mathcal{M}=(Q,q_0,F,A,\delta)$, if $s\not\in Q$ then one can associate
a substitution $\sigma_\mathcal{M}:Q\cup \{s\}\to
(Q\cup\{s\})^*$ defined by 
$$\sigma_\mathcal{M}:\left\{\begin{array}{ll}
s & \mapsto s\, q_0\cr 
q & \mapsto \delta(q,a_1)\cdots \delta(q,a_k),\ \forall q\in Q \cr
\end{array}\right.
$$
where in the last expression, if $\delta(q,a)$ is not defined for
some $a$, then it is replaced by $\epsilon$. Observe that
$\sigma_\mathcal{M}$ can be erasing. This kind of substitution was
introduced for instance in \cite{RM}. The substitution associated to
the canonical automaton of $L$ is said to be the {\it canonical
  substitution} of $L$ and is denoted $\sigma_L$.
\end{defi}

Let $L$ be a regular language and $\sigma_L:B\to B^*$ be its canonical
substitution and let $\tau:A\to A^*$ be a substitution. If there
exists an onto mapping $\Phi:A\to B$ such that for all $a\in A$,
$$\Phi(\tau(a))=\sigma_L(\Phi(a))$$
then $\tau$ is said to be an {\it
  $L$-substitution}. Clearly, if $\mathcal{M}$ is an $L$-automaton then
$\sigma_\mathcal{M}$ is an $L$-substitution.

\begin{prop}\label{prop:numsub}
  Let $S=(L,A,<)$ be a numeration system and $E\subset\mathbb{N}$ be an
  $S$-recognizable set. Then there exists an $L$-substitution
  $\sigma:B\to B^*$ having $x\in B^\omega$ as fixed point
  and a morphism $f:B\to \{0,1\}\cup\{\epsilon\}$ such that
  $$f(x)=\chi_E.$$
\end{prop}

\begin{proof} 
  We denote by $\mathcal{A}=(Q',q_0',F',A,\delta')$ the canonical
  automaton of $L$. By Lemma \ref{lemma:Srec}, $\rep_S(E)$ is accepted
  by some $L$-automaton $\mathcal{M}=(Q,q_0,F,A,\delta)$. Let
  $\Phi:Q\to Q'$ be the mapping related to the $L$-automaton and let
  $F''$ be the set of states of $\mathcal{M}$ given by
  $$F''=\Phi^{-1}(F').$$
  Observe that $F\subseteq F''\subseteq Q$.
  Consider the alphabet $B=Q\cup\{s\}$ ($s\not\in Q$), the
  $L$-substitution $\sigma_\mathcal{M}:B\to B^*$ and the mapping
  $f:B\to \{0,1\}\cup\{\epsilon\}$ defined by
  $$f : \left\{
 \begin{array}{ll} 
s & \mapsto \epsilon; \cr 
q & \mapsto 1, \ {\rm  if}\ q \in F ;\cr 
q & \mapsto 0, \ {\rm if}\ q \in F''\setminus F ;\cr 
q & \mapsto \epsilon,\ {\rm if}\ q \in Q\setminus F''.\cr
\end{array}
\right.$$
It is easy to show that
$\lim_{n\to\infty}f(\sigma_\mathcal{M}^n(s))=\chi_E$.
\end{proof}

\begin{exam}
  We continue Example \ref{exaab}. The canonical automaton of
  $L=a^*b^*$ has two states $A$ and $B$ such that $\delta(A,a)=A$,
  $\delta(A,b)=B$ and $\delta(B,b)=b$. Proceeding as in Definition
  \ref{defisigmam}, we get the substitution
$$\sigma_{\mathcal{M}}:s\mapsto sA,\ A\mapsto AB,\ B\mapsto B$$
having
$$w=sAABABBABBBABBBBABBBBB\cdots$$
as fixed point. Applying the morphism $f:s\mapsto\epsilon, A\mapsto
1,B\mapsto 0$ to this word $w$, we get the characteristic sequence of
the $S$-recognizable set $E_a=\{0,1,3,6,10,\ldots\}$.
\end{exam}

\section{Growth type and erasures}\label{sec4}

\label{preliminary}

In this section, we first consider the growth order of the length of
the iterates of a substitution for any letter. From this we define the
notion of growth order of a letter. Then we give arguments that allow
us to get rid of erasing substitutions. In the third part of this
section, we exhibit sub-alphabets which are invariant for the
substitution. All of these results will play an important r\^ole in
the proof of our main result.

Finally, we consider the relationship of the growth order of the
substitution with abstract numeration systems. This will lead to an
easy adaptation of Cobham's theorem given in terms of substitutions to
these abstract numeration systems.

\subsection{Growth type}\label{sec:growth}

In this subsection we recall some lemmata and definitions appearing in
\cite{Du1}.

Notice that in the following lemma, the substitutions $\sigma$ and
$\tau$ can be erasing. As we will see in the detailed proof of the
result, the technical procedure of replacing $\tau$ with one of its
power, allows us to get rid of irreducible components to the benefit of
irreducible ones.
\begin{lemma}
\label{croissance}
Let $\tau:A\to A^*$ be a substitution on the finite alphabet $A$.
There exists $p$ such that for $\sigma=\tau^ p$ and for all $a\in A$,
one of the following two situations occurs, either
    $$\exists N\in\NN : \forall n>N,\ |\sigma^n(a)|=0,$$
    or there
    exist $d(a)\in\NN$ and algebraic numbers $c(a),\alpha(a)$ such that
    $$
    \lim_{n\to+\infty} \frac{|\sigma^n(a)|}{c(a)\, n^{d(a)}\,
      \alpha(a)^n}=1.$$
    Moreover, if the latter situation occurs then
    for all $i\in\{0,\ldots,d(a)\}$ there exists a letter $b\in A$
    appearing in $\sigma^j(a)$ for some $j\in\NN$ and such that
    $$\lim_{n\to+\infty} \frac{|\sigma^n(b)|}{c(b)\, n^i\,
      \alpha(a)^n}=1.$$
\end{lemma}
\begin{proof}
    With $\sigma$ we associate an automaton $\mathcal{A}_\sigma$ in
    the classical way: the set of states of $\mathcal{A}_\sigma$ is
    $A$, the alphabet is $\{1,\ldots,\max_{a\in A}|\sigma(a)|\}$ and
    the transition function $\delta$ is defined as follows. If $b$
    appears in $\sigma(a)$ at position $i\ge 1$ then $\delta(a,i)=b$.
    Notice that $\delta(a,k)$ is not defined if $k>|\sigma(a)|$. So
    $\mathcal{A}_\sigma$ is possibly not a complete automaton. From
    the definition of $\mathcal{A}_\sigma$, it follows that
    $|\sigma^n(a)|$ is exactly the number of paths of length $n$ in
    $\mathcal{A}_\sigma$ starting from $a$.

    We write $a\to b$ if there exists a path in $\mathcal{A}_\sigma$
    from $a$ to $b$. We define an equivalence relation $\sim$ over $A$
    as follows. We define for all $a,b\in A$,
    $$a\sim b\Leftrightarrow (a=b) \text{ or }(a\to b \text{ and }b\to
    a).$$
    As usual, an equivalence class for $\sim$ is said to be a
    {\em communicating class}. Proceeding as in \cite[p.
    119]{LindMarcus}, the communicating classes and the corresponding
    states of $\mathcal{A}_\sigma$ can be ordered in such a way that
    the matrix associated with $\sigma$ has a block triangular form
    \begin{equation}
        \label{eq:triang}
M_\sigma=
\begin{pmatrix}
M_1&0&0&\ldots&0\\
*&M_2&0&\ldots&0\\
*&*&M_3&\ldots&0\\
\vdots&\vdots&\vdots&\ddots&\vdots\\
*&*&*&\ldots&M_k
\end{pmatrix}.
\end{equation}
We denote by $C_j$ the communicating class related to $M_j$. Each
$M_j\neq 0$ is irreducible. 
Let $p_j$ be the corresponding period
(i.e., the smallest integer $t$ such that $(M_j)^t$ has positive
entries on the main diagonal). Let $p={\rm lcm}_{j=1,\ldots,k}\, p_j$.
Replacing $\sigma$ with $\sigma^p$ does not affect its fixed point.
The communicating classes $C_j$ related to primitive blocks $M_j$ are
the same in $\mathcal{A}_\sigma$ and $\mathcal{A}_{\sigma^p}$ but each
communicating class in $\mathcal{A}_\sigma$ related to a nonzero block
which is not primitive is split into several communicating classes in
$\mathcal{A}_{\sigma^p}$ related to primitive blocks (see for instance
\cite[Section 4.5]{LindMarcus}). Assuming that $\sigma$ has been
replaced by $\sigma^p$ (this has no consequence for the rest of this
paper because we are mainly interested in the fixed points of $\sigma$,
so we may assume that the substitutions we consider have such a
property), we may assume in what follows that each $M_j$'s appearing
in \eqref{eq:triang} is either primitive or zero.  Let $\alpha_j$ be
the Perron-Frobenius eigenvalue associated with $M_j\neq 0$. If
$M_j=0$, we set $\alpha_j=0$. One can already notice that $\alpha_j$
is algebraic since $M_j$ has only integer entries.  Notice also that
$\alpha_j=1\Leftrightarrow M_j=(1)$. The number of words of length $n$
starting from and ending to a state related to $M_j$ is of the form
$\sim c_j\, \alpha_j^n$. Since $c_j$ can be computed from left and
right Perron eigenvectors of $M_j$ (see \cite[Thm
4.5.12]{LindMarcus}), it is clear that $c_j$ is an algebraic number
(computations take place in $\mathbb{Q}(\alpha_j)$).

We now estimate the number $|\sigma^n(a)|$ of paths of length $n$ in
$\mathcal{A}_\sigma$ starting from a given state $a$ belonging to
$C_k$. In the graph of the communicating classes (we use once again
the terminology of \cite[p.  119]{LindMarcus}), consider the set
$\mathcal{P}_k$ of all paths starting in $C_k$ and ending in a leaf. Let
$C_{k,0}=C_k,C_{k,1},\ldots,C_{k,\ell}$ be such a path $\mathfrak{p}$
(we will only consider classes such that $M_{k,i}\neq 0$, if no such a
class exists then the corresponding number of words of length $n$ is
zero for $n$ large enough).  The contribution of $\mathfrak{p}$ to
$|\sigma^n(a)|$ is
$$\sim c_{k,0}\ldots c_{k,\ell}\sum_{n_0+\cdots+n_\ell=n}
\alpha_{k,0}^{n_0}\cdots \alpha_{k,\ell}^{n_\ell}$$

Let $$\beta=\max_{i=0,\ldots,\ell}\alpha_{k,i}$$
and
$C_{k,j_1},\ldots,C_{k,j_t}$ be the communicating classes having
$\beta$ as Perron-Frobenius eigenvalue, $t\ge 1$. Therefore the
contribution of $\mathfrak{p}$ to $|\sigma^n(a)|$ is $$\sim
c_{k,0}\ldots c_{k,\ell}\, n^{t-1}\, \beta^n.$$
In particular, it
follows that the Jordan-decomposition of the incidence matrix of
$\mathcal{A}_\sigma$ restricted to the states occurring in
$\mathfrak{p}$ contains a Jordan block of size $t$ for the eigenvalue
$\beta$. To conclude the first part of the proof, we just have to sum
expressions like the one obtained above for all paths in
$\mathcal{P}_k$.

The particular case is immediate, with the same notation as above, if
$b$ belongs to $C_{k,j_m}$, $m\in\{2,\ldots,t\}$, then the
contribution of $\mathfrak{p}$ to $|\sigma^n(b)|$ is proportional to
$n^{t-m}\, \beta^n$. Moreover, since $a$ belongs to $C_{k,0}$, it is
clear that $a\to b$, i.e., there exists $j$ such that $b$ appears in
$\sigma^j(a)$.
\end{proof}

Notice that the following definition is mainly relevant for
non-erasing substitutions (and the next subsection allows us to only
consider such substitutions).

\begin{defi} Let $\sigma$ be a non-erasing substitution possibly
  replaced by a convenient power as in the proof of the previous
  lemma. For all $a\in A$ we will call {\it growth type} of $a$ the
  couple
$$(d(a), \alpha(a))$$ as introduced in the previous lemma. If
$(d,\alpha)$ and $(e,\beta)$ are two growth types we say that
$(d,\alpha)$ is {\it less than} $(e,\beta)$ (or $(d,\alpha) <
(e,\beta)$) whenever $\alpha < \beta$ or, $\alpha = \beta$ and $d<e$.
\end{defi}

Consequently if the growth type of $a\in A$ is less than the growth
type of $b\in A$ then $\lim_{n\rightarrow +\infty} |\sigma^n (a)|
/|\sigma^n(b)| = 0$.  We say that $a\in A$ is a {\it growing letter}
if
$$(d(a),\theta (a))>(0,1)$$ or equivalently, if $\lim_{n\to +\infty}
|\sigma^n (a)| = +\infty $.

We set 
$$\Theta := \max \{ \theta(a) \mid a\in A \},\quad D := \max \{ d(a) \mid 
\theta(a) = \Theta \mid a\in A \}$$ and $A_{max} := \{a\in A \mid
\theta (a) = \Theta , d(a) = D \}$. The dominant eigenvalue of $M$ is
$\Theta$.  We will say that the letters of $A_{max}$ are {\it of
  maximal growth} and that $(D,\Theta)$ is the {\it growth type} of
$\sigma$. Consequently, we say that a substitutive sequence $y$ is
{\it $(D,\Theta)$-substitutive} if the underlying substitution is of
growth type $(D,\Theta)$.

Observe that if $\Theta=1$, then in view of the last part of Lemma
\ref{croissance}, there exists at least one non-growing letter of
growth type $(0,1)$.  Otherwise stated, if a letter has a polynomial
growth, then there exists at least one non-growing letter.
Consequently $\sigma$ is growing (i.e., all its letters are growing)
if and only if $\theta (a) > 1$ for all $a\in A$.  We define
\begin{center}
\begin{tabular}{ccrcl}
$\lambda_{\sigma}$ & : & $A^*$               & $\rightarrow$ & $\RR$\\
          &   & $u_0\cdots u_{n-1}$ & $\mapsto $    & $\sum_{i=0}^{n-1} c(u_i){\bf 1}_{A_{max}} (u_i)$,
\end{tabular}
\end{center}
where $c:A \to \RR_+$ is defined in Lemma \ref{croissance} and ${\bf
  1}_A$ is the usual characteristic function of the set $A$.  From
Lemma \ref{croissance} we deduce the following lemma.
\begin{lemma}
\label{lambda}
For all $u\in A^{*}$ we have $\lim_{n\rightarrow +\infty} |\sigma^n
(u)|/n^D\Theta^n = \lambda_{\sigma} (u)$.
\end{lemma}

We say that the word $u\in A^{*}$ is of {\it maximal growth} if
$\lambda_{\sigma} (u) \not = 0$.

\begin{coro}
\label{puissance}
For all $k\geq 1$, the growth type of $\sigma^k $ is $(D, \Theta^k )$.
\end{coro}

\subsection{Erasing morphisms}

In view of Proposition \ref{prop:numsub}, we will have to deal with
erasing substitutions and also with erasing morphisms. The following
two propositions show how to get rid of the erasing behavior.

\begin{prop}
\label{nonerasing}
Let $x$ be a proper fixed point of a substitution $\sigma : A\to A^{*}$ with
 growth type $(D , \Theta )$. Then, there exists a non-erasing substitution
 $\tau : C \to C^*$ with a proper fixed point $y$, a letter-to-letter morphism $\psi : C\to A$ and a morphism $\phi : A \to C^*$ verifying 
\begin{enumerate}
\item
$x = \psi (y) $;
\item There exists $l\in \NN$ such that for all $n\in \NN$ we have
  $\tau^n \circ \phi = \phi \circ \sigma^{ln}$;
\item
Each line and each column of the matrix of $\phi$ has a non-zero coefficient; 
\item
The growth type of $\tau$ is $(D , \Theta^l)$. 
\end{enumerate}
\end{prop}

\begin{proof}
  The statement (1), (2) and (3) can be found in \cite[Theorem 7.5.1,
  p. 227]{AS} and (4) is a consequence of (2) and (3).
\end{proof}

\begin{prop}
\label{nonerasing2}
Let $x$ be a proper fixed point of a substitution $\sigma : A\to
A^{*}$ with growth type $(D , \Theta )$, $\Gamma \subset A$ and $\zeta
: A \to A\setminus \Gamma$ defined by $\zeta (a) = \epsilon$ if $a\in
\Gamma$ and $a$ otherwise.  Then, there exists a non-erasing
substitution $\tau : C \to C^*$ with a proper fixed point $y$, a
letter-to-letter morphism $\psi : C\to A$ and a morphism $\phi : A \to
C^*$ verifying
\begin{enumerate}
\item
$\zeta (x) = \psi (y) $;
\item 
There exists $l\in \NN$ such that for all $n\in \NN$ we have
  $\psi \circ \tau^n \circ \phi = \zeta \circ \sigma^{l(n+1)}$;
\item
Each line and each column of the matrix of $\phi$ has a non-zero coefficient; 
\item
The growth type of $\tau$ is $(D , \Theta^l)$. 
\end{enumerate}
\end{prop}

\begin{proof}
The statement (1), (2) and (3) can be found in \cite[pp. 232--236]{AS} and (4) is a
 consequence of (2) and (3).
\end{proof}

\subsection{Invariant alphabets}

Let $\Delta(w)\subseteq A$ be the set of letters having an occurrence
in the word $w\in A^*$.

\begin{lemma}
\label{stabilisation}
Let $\sigma:A\to A^*$ be a non-erasing substitution.  There exists $N\ge 1$
such that for all $a\in A$ and all $n\ge 1$,

$$
\Delta((\sigma^N)^n(a))=\Delta(\sigma^N(a)).
$$
\end{lemma}

\begin{proof}
  We set $A=\{ a_1, \dots, a_{|A|} \}$.  The alphabet $A$ being
  finite, the sequence of sub-alphabets $(\Delta (\sigma^n
  (a_1)))_{n\in \NN}$ is ultimately periodic, i.e., there exist $p$
  and $q$ such that $$\Delta (\sigma^{q+np+i} (a_1)) = \Delta
  (\sigma^{q+mp+i} (a_1))$$ for all $m,n,i\in \NN$.  Hence, for $k$
  such that $kp\geq q$, we have for all $n\ge 1$

$$
\Delta((\sigma^{kp})^n(a_1))=\Delta(\sigma^{kp}(a_1)).
$$
Now take $a_2$ and consider $\sigma^{kp}$. 
Proceeding as before we find $r$ such that 
$$
\Delta((\sigma^{r})^n(a_1))=\Delta(\sigma^{r}(a_1)) \hbox{ and } \Delta((\sigma^{r})^n(a_2))=\Delta(\sigma^{r}(a_2)) .
$$

We conclude continuing like this with $a_3$, $\dots$, $a_{|A|}$.
\end{proof}

The following corollary is just a reformulation of the previous lemma.
\begin{coro}
Let $\sigma:A\to A^*$ be a non-erasing substitution. 
There exists $N\ge 1$ such that for all $a,b\in A$ and $n\ge 1$
$$
a\in A \hbox{ appears in } (\sigma^N)^n (b) \hbox{ if and only if } a \hbox{ appears in } (\sigma^N)^{n+1} (b) .
$$
\end{coro}
Replacing $\sigma$ by one of its power $\sigma^N$ does not alter its
fixed points (we will use this argument repeatedly). Therefore we will
often require that $\sigma$ has the following property:
\begin{align}
\label{stab1}
\forall a,b\in A,\ \forall n\ge 1,\quad a\in A \hbox{ appears in } \sigma^n (b) \hbox{ if and only if } a \hbox{ appears in } \sigma^{n+1} (b) .
\end{align}

\subsection{Linking the growth order with numeration systems}\label{sublink}

Let $L$ be a regular language having $\A=(Q',q_0',F',A,\delta')$ as
canonical automaton and $M_\A$ as associated matrix.
\begin{enumerate}
\item As in section \ref{sec:growth}, for all states $q'\in Q'$ we can
  define the growth type $(d,\alpha)$ of $q'$ (corresponding to the
  number of words of length $n$ accepted in $\A$ from $q'$) and
  consequently, we can define the growth type of $\A$ as the largest
  growth type of the states in $Q'$.
\item If $\M=(Q,q_0,F,A,\delta)$ is an $L$-automaton then $\M$ and $\A$
  have the same growth type. Indeed, for any $q'\in Q'$, we denote by
  $p_{\A,q'}(n)$ the number of paths of length $n$ in $\A$ starting in
  $q'$. If $\Phi:Q\to Q'$ is the mapping defining the $L$-automaton,
  then for any $q\in\Phi^{-1}(q')$, 
  $$p_{\A,q'}(n) \ge p_{\M,q}(n)$$
  and also
  $$p_{\A,q'}(n)\le \sum_{q\in\Phi^{-1}(q')} p_{\M,q}(n).$$
  This means
  that $q'$ and at least one of the states $q\in\Phi^{-1}(q')$ are of
  the same growth type and that none of the states $q\in\Phi^{-1}(q')$
  is of a larger growth type than $q'$.
  \item If $\M$ is of growth type $(D,\Theta)$, $\Theta>1$, then
    $\sigma_\M$ is of the same growth type. But notice that if $\M$ is
    of growth type $(D,1)$ then $\sigma_\M$ is of growth type $(D+1,1)$.
\end{enumerate}
As a consequence of theses observations, if $L$ is a regular language
having a canonical automaton of growth type $(D,\Theta)$, $\Theta>1$,
(resp. $(D,1)$, $D\ge 1$) and if $E\subset\mathbb{N}$ is
$S$-recognizable for the numeration system $S=(L,A,<)$ then from
Propositions \ref{prop:numsub}, \ref{nonerasing} and \ref{nonerasing2}
the sequence $\chi_E$ is $(D,\Theta^l)$-substitutive for some $l$
(resp. $(D+1,1)$-substitutive). This obersevation will be helpful in
the last section of this paper (Corollary \ref{abs1} and Remark
\ref{abs2}).

\section{The words appear with bounded gaps}

This section is devoted to the proof of the main result of this paper:

\begin{theo}\label{the:boundedgap}
Let $d,e\in \NN\setminus \{ 0 \}$ and $\alpha , \beta \in [ 1,+\infty [$ such that 
$(d,\alpha ) \not = (e,\beta )$ and satisfying one of the following three conditions:

\begin{enumerate}
\item $\alpha $ and $\beta$ are multiplicatively independent;
\item $\alpha , \beta > 1$ and $d\not = e$;
\item $(\alpha,\beta)\neq (1,1)$ and, $\beta = 1$ and $e\not = 0$, or,
  $\alpha = 1$ and $d\not = 0$;
\end{enumerate}

Let $C$ be a finite alphabet.  If $x\in C^\NN$ is both $(d, \alpha
)$-substitutive and $(e, \beta )$-substitutive then the letters of $C$
which have infinitely many occurrences in $x$ appear in $x$ with
bounded gaps.
\end{theo}

For the proof of this result we will proceed into three parts. The
first part consists of arithmetical lemmata about density in $\RR$.
In the second part we give bounds for gaps created by some letters. In
subsections \ref{53} and \ref{54} we exhibit an important sequence of
integers and we fix some useful constants. Finally from subsection
\ref{55} to \ref{finn} we proceed to a case study depending on the
growth order of the considered substitutions. Let us first fix the
context we will be dealing with.

\medskip

Let $\sigma$ and $\tau$ be two substitutions on the alphabets $A$ and
$B$, with fixed points $y$ and $z$ and with growth types $(d ,
\alpha)$ and $(e , \beta )$ respectively.  Taking powers of $\sigma$
and $\tau$ does not alter the fixed points $y$ and $z$ and does not
change the multiplicative dependence.  Thus, in the proof we will
sometimes replace the substitution by some convenient power of
itself (and this also allows us to assume that condition \eqref{stab1}
is satisfied). In particular, when $\alpha$ and $\beta$ are
multiplicatively dependent we may suppose that $\alpha = \beta$.
 
Let $\phi : A\rightarrow C$ and $\psi : B\rightarrow C $ be two letter-to-letter morphisms such that $\phi (y) = \psi (z) = x$.  Lemma
\ref{nonerasing} allows us to suppose that $\sigma$ and $\tau$ are
non-erasing.  We call $A_+$ the set of growing letters of $A$ with
respect to $\sigma$.

\subsection{Some density lemmata}

Recall that $\alpha, \beta \in [1,+\infty [$ are {\it multiplicatively
  independent} whenever $\alpha^k = \beta^\ell$, $\ell,k\in \NN $,
implies $k=0$ or $\ell=0$.  In \cite[Corollary 11]{Du1} the following
result is proved. Observe that this result is well known when $d=e=0$
(and is sometimes stated as a Kronecker's theorem). Moreover it does
not take into account the case $\alpha=1$ or $\beta=1$.

\begin{theo}
\label{midensity}
Let $\alpha$ and $\beta$ be multiplicatively independent elements of $]1,+\infty[$. 
Let $d$ and $e$ be non-negative integers. Then the set
$$
\left\{ 
\frac{\alpha^n n^d}{\beta^m m^e} ; n,m \in \NN
\right\}
$$
is dense in  $\RR_+$.
\end{theo}

\begin{lemma}
\label{lemmadensity}
Let $d,e\in \NN$ and $\alpha \in ] 1,+\infty [$. Then,
\begin{enumerate}
\item
\label{ld1}
$d,e \geq 1$ if and only if the set
$
\left\{ 
\frac{n^d}{m^e} ; n,m \in \NN
\right\}
\hbox{ is dense in } \RR_+;
$
\item
\label{ld2}
$e\not = 0$ if and only if
$
\left\{ 
\frac{\alpha^n n^d}{m^e} ; n,m \in \NN
\right\}
\hbox{ is dense in } \RR_+;
$
\item
\label{ld3}
$d\not = e$ if and only if 
$
\left\{ 
\frac{\alpha^n n^d}{\alpha^m m^e} ; n,m \in \NN
\right\}
\hbox{ is dense in } \RR_+ .
$ 
\end{enumerate}
\end{lemma}

\begin{proof}
\eqref{ld1} Suppose $d,e\geq 1$. Let $l\in \RR_+\setminus \{ 0 \}$ and $\epsilon >0$. It suffices to find $n,m\in \NN$ such
 that $|l -n^d/ m^e | < \epsilon $.

 Let $m\in \NN$ be such that $\max (d, 2^d l/\epsilon) < (lm^e)^{1/d}
 - 1$ and $1/m^e < l$.  There exists $n\in \NN$ such that $n^d / m^e <
 l \leq (n+1)^d/m^e$.  We observe that this implies that $n> d$ and
 $2^dl/n < \epsilon$.  Consequently, we get

$$
0
< l - \frac{n^d}{m^e} 
\leq \frac{(n+1)^d - n^d}{m^e} 
\leq \frac{2^d n^{d-1}}{m^e}
= \frac{2^d}{n} \frac{n^d}{m^e}
<\frac{2^d l}{n} < \epsilon.
$$

Hence the set $\{ n^d / m^e ; n,m \in \NN \} $
is dense in $\RR_+$.

\medskip

\eqref{ld2} Suppose $e\not =0$. Let $l\in \RR_+\setminus \{ 0 \}$ and $\epsilon >0$. It suffices to find $n,m\in \NN$ such
 that $|l -\alpha^n n^d/ m^e | < \epsilon $.

Let $m_0\in \NN$ be such that
$
e\ln (1+1/m_0) < \ln (1+ \epsilon/l) .
$
Let $n$ be such that $e\ln (m_0) < d\ln (n) + n\ln (\alpha) - \ln (l)$
 and $m\geq m_0$ be such that $e\ln (m) \leq d\ln (n) + n\ln (\alpha) -\ln (l) \leq e\ln (m+1)$. Then we have

$$
0
\leq
d\ln (n) + n\ln (\alpha) -\ln (l) -e\ln (m)
\leq 
e\ln \left(1+\frac{1}{m}\right) 
< \ln \left(1+ \frac{\epsilon}{l}\right) .
$$

Hence the set $\{ \alpha^n n^d / m^e ; n,m \in \NN \} $ is dense in
$\RR_+$.

\medskip

\eqref{ld3} Suppose $d\not = e$.
Let $l\in \RR_+\setminus \{ 0 \}$ and $\epsilon > 0$. It suffices to find $n,m\in \NN$ such
 that $|l -n^d \alpha^n / m^e \alpha^m | < \epsilon $.

 We can suppose $d> e$ because $\{ \alpha^n n^d / \alpha^m m^e ; n,m
 \in \NN \}$ is dense in $\RR_+$ if and only if $\{ \alpha^m m^e /
 \alpha^n n^d ; n,m \in \NN \}$ is dense in $\RR_+$.

Let $n_0$ be such that 

$$
\frac{d-e}{2\ln \alpha } \ln n_0 
\leq 
\frac{d-e}{\ln \alpha }\ln n_0 - \frac{\ln (l+\epsilon)}{\ln \alpha}
\leq 
\frac{d-e}{\ln \alpha }\ln n_0 - \frac{\ln (l)}{\ln \alpha}
\leq
\frac{3(d-e)}{2\ln \alpha } \ln n_0 .
$$ 

Choose $b_0$ with $(\epsilon \alpha^{b_0})^{\frac{1}{d-e}} \geq 1$.
Then for all $b\geq b_0$ there exists $n_b$ such that 

$$
l\alpha^b 
\leq 
n_b^{d-e} 
\leq 
(l+\epsilon ) \alpha^b .
$$

The sequence $(n_b)$ goes to infinity, 
consequently we can choose $b$ and $n$ such that $n=n_b \geq n_0$ and $1-\epsilon/l \leq (n/n+b)^e $.
Then we have

$$
\frac{d-e}{2\ln \alpha } \ln n
\leq 
b
\leq
\frac{3(d-e)}{2\ln \alpha } \ln n
$$

Now consider $m=n+b$. This gives

$$
l-\epsilon 
\leq
l\left( \frac{n}{b+n} \right)^e
\leq
\frac{n^d \alpha^n}{m^e \alpha^m} 
\leq
(l+\epsilon )\left( \frac{n}{b+n} \right)^e 
\leq
l+\epsilon.
$$

\medskip

Suppose $d=e$. If $n\leq m$ then $\alpha^n n^d / \alpha^m m^e \leq 1$
 and if $n> m$ then  $\alpha^n n^d / \alpha^m m^e \geq \alpha$. This
 concludes the proof.
\end{proof}

\begin{coro}
\label{density}
Let $d,e\in \NN$ and $\alpha , \beta \in [ 1,+\infty [$. 
We set

$$
\Omega = \left\{ 
\frac{\alpha^n n^d}{\beta^m m^e} ; n,m \in \NN
\right\} .
$$ 

Then $\Omega$
is dense in $\RR_+$ if and only if one of the following two  conditions
 holds:

\begin{enumerate}
\item
$\alpha$ and $\beta$ are multiplicatively independent.
\item
$\alpha , \beta > 1$ and $d\not = e$.
\item
$\beta = 1$ and $e\not = 0$, or, $\alpha = 1$ and $d\not = 0$;
\end{enumerate}
\end{coro}

\begin{proof}
It follows from Theorem \ref{midensity} and Lemma \ref{lemmadensity}.
\end{proof}

We will say that two substitutions are {\it independent} whenever their respective growth type $(d,\alpha)$ and $(e,\beta )$
are different and satisfy Hypothesis (1), (2) or (3) in the previous corollary.
Notice that in Theorem \ref{the:boundedgap} the assumptions mean that the substitutions are independent and are not both 
of polynomial growth (this corresponds to the hypothesis $(\alpha , \beta ) \not = (1,1)$).

\subsection{Growth type of gaps}

In this subsection we give two results on the gaps created by the
letters of some sub-alphabet in prefixes of fixed points and in
iterates of letters.  They will be key arguments in the proof of
Theorem \ref{the:boundedgap}.

Let $E\subset A$.  For all $N\ge 1$, we set
$$
M(N , x , E) := \max \{ k\in\NN : \exists i\in [0,N-k+1] ,
|x_{[i,i+k]}|_{E} = k \} .
$$
In what follows, if $x$ and $E$ are clear from the context, we simply
write $M(N)$.

\begin{prop}
\label{gap-prefix}
Let $x=(x_n)_{n\ge 0}$ be a proper fixed point of the non-erasing
substitution $\sigma$ of growth type $(d,\alpha)$ on the finite alphabet $A$. Assume  $\sigma$
is such that each letter of $A$ has an occurrence in $\sigma (x_0)$
and $\sigma$ satisfies \eqref{stab1}.

Let $E\subset A$.  Suppose there exists a letter $e\in A$ such that
$\sigma (e) \in E^*$ and call $E'$ the set of all such letters.  Let
$(d' , \alpha' )$ be the greatest growth order among the elements of
$E'$.  Then, in each of the following situations, there exist two
constants $C_1,C_2>0$ such that

\begin{enumerate}
\item If $(\alpha' , d' ) = (\alpha , d)$ then, for all $N$,
$$
C_1 N\leq M(N)=M(N,x,E) \leq C_2 N.
$$
\item If $\alpha = \alpha' > 1 $ and $d'<d$ then, for all $N$,
$$
C_1 N (\log N)^{d'-d} \leq M(N) \leq C_2 N (\log N)^{d'-d}.
$$
\item If $\alpha> \alpha' >1$ then, for all $N$,
$$ 
C_1 (\log N)^{d'-d\frac{\log \alpha'}{\log \alpha}}N^{\frac{\log \alpha'}{\log \alpha}} 
\leq  M(N) \leq C_2 (\log N)^{d'-d\frac{\log \alpha'}{\log \alpha}}N^{\frac{\log \alpha'}{\log \alpha}}.
$$
\item If $\alpha> \alpha' =1$ then for all $N$,
$$
C_1 \left(\frac{\log N}{\log \alpha}\right)^{d' } \leq M(N) \leq C_2
\left(\frac{\log N}{\log \alpha}\right)^{d' + 1} .
$$
\item If $\alpha = \alpha' = 1$ and $d' < d$ then for all $N$
$$
C_1 N^{d'/d} \leq M(N) \leq C_2 N^{(d'+1)/d} .
$$
\end{enumerate}

\end{prop}

\begin{rema}
  Notice that as usual the assumptions on $\sigma$ made in the
  statement of Proposition \ref{gap-prefix} are easily satisfied by
  taking a convenient power of $\sigma$ if needed.
\end{rema}

\begin{proof}
Let $N\in \NN$. 
There exists $n\in \NN$ such that 

\begin{equation}
\label{petitn}
|\sigma^{n-1} (x_0) |\leq N\leq | \sigma^{n} (x_0) | .
\end{equation}

We start proving (1). As there exists a letter $e\in E'$ of maximal
growth having an occurrence in $\sigma (x_0)$ (and since \eqref{stab1}
is satisfied, $\sigma^ k(e)\in E^*$, for all $k\ge 1$) we obtain

$$
|\sigma^{n-2} (e) | \leq M(N) \leq |\sigma (x_0)|\max_{l\in A} |\sigma^{n-1} (l)|
$$

and from Lemma \ref{croissance} we deduce that there exist two
constants $C_1$ and $C_2$ not depending on $n$ such that

$$
C_1 \alpha^n n^d \leq M(N) \leq C_2 \alpha^n n^d .
$$

Let us prove (2).  The assertion (3) can be proved following the same
arguments.  We start proving the left inequality.  Proceeding as
before we obtain a constant $C'_1$ depending neither on $n$ nor $N$
such that

$$
C'_1 \alpha^n n^{d'} \leq M(N) .
$$

Moreover, from \eqref{petitn}, we deduce there exist two constants
$C''_1, C''_2$ depending neither on $n$ nor $N$ such that

\begin{align}
\label{temp}
C''_1 \log (N) \leq n \leq C''_2 \log (N) &  \hbox{ if } \alpha >1 \hbox{ and } \\
\label{temp2}
C''_1 N^{1/d} \leq n \leq C''_2 N^{1/d} &  \hbox{ if } \alpha =1 . 
\end{align}
 
This together with Lemma \ref{croissance} gives the left inequality.

Let us prove the right inequality.  Let $i$ be such that
$|x_{[i,i+M(N) ]}|_{E} = M(N)$.  We set $u = x_{[i,i+M(N) ]}$.  There
exist $u_1 \in E'^*$ having an occurrence in $\sigma^{n-1} (x_0)$ and
$p_1,s_1 \in E^*$ such that $u=s_1\sigma (u_1) p_1$ and $|s_1|$,
$|p_1|$ less than $m = \max\{ |\sigma (a)| ; a\in E \}$.  In the same
way there exist $u_2 \in A^*$ having an occurrence in $\sigma^{n-2}
(x_0)$ and $p_2,s_2 \in E'^*$ such that $u_1=s_2\sigma (u_2) p_2$ and
$|s_2|$, $|p_2|$ less than $m$.  We remark that $\sigma^2 (u_2)$
belongs to $E^*$.  From Hypothesis \eqref{stab1}, we conclude that
$u_2$ belongs to $E'^*$.  Hence there exist $u_1,\dots , u_{n-1} \in
E'^*$, $p_1,s_1 \in E^*$, $p_2, \dots , p_{n-1}, s_2,\dots , s_{n-1}
\in E'^* $ such that

\begin{equation}
\label{decomposition}
u = s_1 \sigma (s_2) \cdots \sigma^{n-2} (s_{n-1}) \sigma^{n-1} (u_{n-1}) \sigma^{n-2} (p_{n-1} )\cdots \sigma (p_2) p_1 ,
\end{equation}

$|p_i|$ and $|s_i|$ are less than $m$.  From this expression and Lemma
\ref{croissance} we deduce that there exists a constant $C'''_2$ such
that

\begin{equation}
\label{calcul-decomp}
|u| \leq C'''_2 \alpha^n n^{d'}.
\end{equation}

We conclude using \eqref{petitn} (together with Lemma \ref{croissance}) and \eqref{temp}.

\medskip

We now prove (4).
For the left inequality it works as before.
For the right inequality it also works as before except that once we obtain the decomposition \eqref{decomposition}
we find some constant $C$ such that $|u|\leq C \sum_{j=1}^{n-1} j^{d'}$. 
Consequently for some other constant $|u|\leq C n^{d'+1}$.
We conclude using Lemma \ref{croissance} and \eqref{temp}

\medskip

For (5) we proceed as in the previous case except we use \eqref{temp2}.
\end{proof}

We suppose there exists a letter $c\in C$ with infinitely many
occurrences in $x$ and that does not appear with bounded gaps in $x$.
Projecting to $\{ 0,1 \}$ we can suppose $C=\{ 0,1 \}$ and $c=1$.
W.l.o.g. we may assume that $\sigma$ and $\tau$ both satisfy
\eqref{stab1} (as usual taking a power of the substitution does not
alter its fixed points).  There exist $a\in A$ with infinitely many
occurrences in $y$ and a strictly increasing sequence $(p_n)_{n\in
  \NN}$ of positive integers such that the letter $c$ does not appear
in $\phi (\sigma^{p_n} (a) )$. Let $A(c)$ be the set of such letters.
We define $B(c)$ and $B_+$ as $A(c)$ and $A_+$ but with respect to
$\tau$ and $B$.

The sets $A(c)$ and $B(c)$ are non-empty.
Then, there exist a letter $a\in A(c)\cap A_+$ and a letter $b\in B(c)\cap B_+$ 
having infinitely many occurrences in $y$ and $z$, with growth type 
$(d^{'},\alpha^{'})\leq(d,\alpha)$ and $(e^{'},\beta^{'})\leq(e,\beta)$, 
respectively, 
where $(d^{'},\alpha^{'})$ and $(e^{'},\beta^{'})$ are maximal with respect to $A(c)$ and $B(c)$.

Because $M(N,y, \phi^{-1} (\{ 0\})) = M(N, z ,\psi^{-1} (\{ 0\}) )$,
from Proposition \ref{gap-prefix} we deduce that we have necessarily one of the following five situations:

\begin{equation}
  \label{eq:fivecases}
\left\{\begin{array}{l}
  (\alpha' ,d') = (\alpha , d) \ \text{ and }\  (\beta' ,e') = (\beta ,e);\\
  \alpha = \alpha'>1,\ \beta'=\beta \ \text{ and }\  d-d' = e-e';\\
\alpha > \alpha' >1\ \text{ and }\ \beta > \beta' >1;\\
\alpha > \alpha' =1\ \text{ and }\ \beta > \beta' =1;\\
\alpha = \alpha' =1,\ d'<d\ \text{ and }\ \beta = \beta' =1,\ e'<e.\\
\end{array}\right.
\end{equation}

We will consider these cases separately.
Before we establish some general facts that will be used in the treatment of these cases.

Let $w=w_0 \cdots w_n$ be a word belonging to $L(y)$ (resp. $L(z)$), 
we call $\gap (w)$ the largest integer $k$ such that 
there exists $i\in [0,n-k+1]$ for which the letter $c$ does not appear 
in $\phi (w_i \cdots w_{i+k-1})$ (resp. in $\psi (w_i \cdots w_{i+k-1})$).

The next lemma is stated for $\sigma$ but of course it also holds for
$\tau$.  Moreover we can assume the constant $K'$ is the same for the
two substitutions.

\begin{lemma}
\label{Kprime}
With notation introduced before, there exists a constant $K^{'}$ such
that for all $a''\in A$ we have:

\begin{align*}
\gap(\sigma^{n} (a'')) & \leq K^{'} n^{d'} {\alpha^{'}}^{n}   \hbox{ if } \alpha' > 1 \hbox{ and }\\
\gap(\sigma^{n} (a'')) & \leq K^{'} n^{d'+1}    \hbox{ if } \alpha' = 1
\end{align*}

for all $n\in \NN$.
\end{lemma} 

\begin{proof}
It suffices to proceed as we did before to obtain \eqref{decomposition} and then \eqref{calcul-decomp}.
\end{proof}

From Lemma \ref{croissance}, the following limits exist and are finite
and they deserve specific notation
$$\lim_{n\rightarrow +\infty }\frac{|\sigma^{n}
(a)|}{n^{d'}{\alpha'}^n}=:\mu (a) \quad\text{ and }\quad
\lim_{n\rightarrow +\infty }\frac{|\tau^{n} (b)|}{n^{e'}{\beta'}^n}=:\mu
(b).$$

\subsection{Some choices when $\alpha, \beta >1$}\label{53}

Here we suppose that the set $\Omega$ of Corollary \ref{density} is dense in $\RR_ +$.
There exist infinitely many prefixes of $y$ (resp. $z$) of the type $u_1 au_2a'$
(resp. $v_1 b v_2b'$) 
fulfilling the conditions $\imath)$ and $\imath\imath)$ below:

\medskip

$\imath$) 
The growth type of $u_1\in A^*$ and $a'\in A$ (resp. $v_1\in B^*$ and $b'\in B$) is maximal (Lemma \ref{croissance} allows such a configuration).

$\imath\imath$) 
The word $u_2$ (resp. $v_2$) does not contain a letter of maximal growth.

\medskip

We notice this is not the case when the growth type is $(d,1)$ because
in this case there is exactly one letter of growth type $(d,1)$ and it
appears exactly once in the fixed point: this is the first letter of
the fixed point.

Let $u_1au_2a'$ be a prefix of $y$ and $v_1bv_2b'$ be a prefix of $z$
fulfilling the conditions $\imath$) and $\imath\imath $).

From Corollary \ref{density} there exist four strictly increasing
sequences of integers $(m_i)_{\in \NN}$, $(n_i)_{\in \NN}$,
$(p_i)_{\in \NN}$ and $(q_i)_{\in \NN}$ such that

\begin{align}
\label{multind1}
\lim_{i\rightarrow +\infty}  \frac{n_i^d \alpha^{n_i}}{m_i^e\beta^{m_i}}  
= & 
\frac{2\lambda_{\tau} (v_1)}{2 \lambda_{\sigma}(u_1)  + 2 \lambda_{\sigma}(a) + \lambda_{\sigma}(a')} =: \gamma_1
\ \
{\rm and} \\
\label{multind2}
\lim_{i\rightarrow +\infty}  \frac{p_i^e\beta^{p_i}}{q_i^d \alpha^{q_i}} 
=&
\frac{2\lambda_{\sigma} (u_1)}{2\lambda_{\tau} (v_1) + 2\lambda_{\tau} (b)  + \lambda_{\tau}(b')} =: \gamma_2.
\end{align}

As a consequence of \eqref{multind1} and \eqref{multind2}, we have 

\begin{equation}
\label{nu} 
\lim_{i\rightarrow +\infty} \frac{n_i}{m_i} = \frac{\log\beta}{\log\alpha} \quad \text{ and }\quad
\lim_{i\rightarrow +\infty} \frac{p_i}{q_i} = \frac{\log\alpha}{\log\beta}.
\end{equation}

The sequence $z$ has infinitely many occurrences of letters of maximal growth.
Hence, in this case, we can take $v_1$ so long that

\begin{align}
\label{gammadeux}
\frac{2K' (2\gamma_2)^{\frac{\log \alpha'}{\log \alpha }}}{\mu (a)}
\cdot
\left(
\frac{\log \alpha}{\log \beta}
\right)^{e^{'} -e\frac{\log \alpha'}{\log \alpha}} 
< 1.
\end{align}

Using Lemma \ref{lambda} there exists $i_0$ such that for all $i\geq i_0$ we have
\begin{align}
\label{multind1bis}
\frac{|\sigma^{n_i}(u_1 a u_2)|}{|\tau^{m_i}(v_1)|} 
& \leq 1 \leq 
\frac{|\sigma^{n_i}(u_1 a u_2a')|}{|\tau^{m_i}(v_1 b v_2)|}
\ \ {\rm and} \\
\label{multind2bis}
\frac{|\tau^{p_i}(v_1 b v_2)|}{|\sigma^{q_i}(u_1)|} 
& \leq 1 \leq 
\frac{|\tau^{p_i}(v_1 b v_2b')|}{|\sigma^{q_i}(u_1 a u_2)|} .
\end{align}

It comes that the word $\psi (\tau^{m_i}(b v_2))$ (resp. $\phi (\sigma^{q_i} (a u_2))$) 
has an occurrence in $\phi (\sigma^{n_i}(a'))$ (resp. $\psi (\tau^{p_i}(b'))$).
To obtain a contradiction it suffices to have some $j\geq i_0$ such that 
$
\gap(\sigma^{n_j} (a'))/ \gap (\tau^{m_j} (b)) < 1 \ \ {\rm or} \ \ \gap(\tau^{p_j} (b'))/ \gap (\sigma^{q_j} (a))< 1.
$

We observe that $\gap (\tau^{m_j} (b)) = |\tau^{m_j} (b)|$ and $\gap (\sigma^{q_j} (a)) =|\sigma^{q_j} (a)|$.
We set $S_j = \gap(\sigma^{n_j} (a'))/ |\tau^{m_j} (b)|$ and $T_j = \gap (\tau^{p_j} (b'))/ |\sigma^{q_j} (a)|$.
Then,

\begin{align}
\label{contradiction}
\hbox{ it suffices to find some $j$ with  } S_j <1 \hbox{ or } T_j < 1 .
\end{align}

We have
\begin{align}
\label{ineg-S1}
S_i
\leq &
\frac{K' n_i^{d'}{\alpha^{'}}^{n_i}}{\mu (b) m_i^{e'} {\beta^{'}}^{m_i}} 
\cdot
\frac{\mu (b) m_i^{e'} {\beta^{'}}^{m_i}}{|\tau^{m_i} (b)|}
\leq
\frac{2K'}{\mu (b)} 
\cdot
\frac{n_i^{d'}\left( \alpha^{n_i} \right)^{\frac{\log \alpha'}{\log \alpha}}}{m_i^{e'}{\beta^{'}}^{m_i}}\\
\label{inegS-2}
\leq &
\frac{2K'}{\mu (b)}
\cdot
\frac{n_i^{d'}}{m_i^{e'}{\beta^{'}}^{m_i}}
\cdot
\left( 2\gamma_1 \frac{m_i^e \beta^{m_i}}{n_i^d} \right)^{\frac{\log \alpha'}{\log \alpha}} \\
\label{inegS-3}
\leq &
\frac{2K' (2 \gamma_1)^{\frac{\log \alpha'}{\log \alpha}}}{\mu (b)}
\cdot
\frac{n_i^{d^{'} -d\frac{\log \alpha'}{\log \alpha}}}{m_i^{e^{'} -e\frac{\log \alpha'}{\log \alpha}}  }
\cdot
\exp 
\left(
m_i 
\left(
\frac{\log \alpha'}{\log \alpha }\log \beta -\log \beta'  
\right)
\right)
\end{align}

and, with the same kind of computations

\begin{align}
\label{ineg-T1}
T_i
\leq &
\frac{K' p_i^{e^{'}}\beta'^{p_i}}{\mu (a) q_i^{d'} {\alpha^{'}}^{q_i}} 
\cdot
\frac{\mu (a) q_i^{d^{'}} {\alpha^{'}}^{q_i}}{|\sigma^{q_i} (a)|}
\leq
\frac{2K'}{\mu (a)} 
\cdot
\frac{p_i^{e'}\beta'^{p_i}}{q_i^{d'}(\alpha^{q_i})^{\frac{\log \alpha'}{\log \alpha}}}\\
\label{ineg-T3}
= &
\frac{2K' (2\gamma_2)^{\frac{\log \alpha'}{\log \alpha}}}{\mu (a)}
\cdot
\frac{p_i^{e^{'} -e\frac{\log \alpha'}{\log \alpha}}}{q_i^{d^{'} -d\frac{\log \alpha'}{\log \alpha}}  }
\cdot
\exp
\left(
p_i 
\left(
\log \beta' - \frac{\log \alpha'}{\log \alpha }\log \beta   
\right)
\right) .
\end{align}

\subsection{Remarks when $\alpha$ and $\beta$ are multiplicatively independent}\label{54}

In this case we necessarily have $\alpha >1$ and $\beta >1$. There
exists $K\geq 2$ and $j_0$ such that for all $i\geq j_0$ we have

\begin{align*}
\frac{1}{K}
\leq
\frac{n_i}{m_i} 
\leq K 
& , \quad
\frac{1}{K}
\leq
\frac{p_i}{q_i} 
\leq K, \\
\frac{n_i^d \alpha^{n_i}}{m_i^e\beta^{m_i}}
\leq
2 \gamma_1
& , \quad
\frac{p_i^e \beta^{p_i}}{q_i^d\alpha^{q_i}}
\leq 2 \gamma_2, \\
\frac{\mu (a) q_i^{d'}{\alpha'}^{q_i}}{|\sigma^{q_i} (a)|}
\leq 2
& , \quad
\frac{\mu (b) m_i^{e'}{\beta'}^{m_i}}{|\tau^{m_i} (b)|}
\leq
2.
\end{align*}
In the sequel we intensively use the previous inequalities and Lemma
\ref{Kprime}.  We can now proceed to a case study. In view of the
hypothesis of Theorem \ref{the:boundedgap}, we will not consider the
last case occurring in \eqref{eq:fivecases}. Subsections \ref{55} to
\ref{finn} correspond to these first four cases.

\subsection{$\boxed{(\alpha' ,d') = (\alpha , d) $ and $(\beta' ,e') = (\beta ,e)}$}\label{55}

We necessarily have $\alpha , \beta >1$.  From \eqref{ineg-T3} we get

\begin{align*}
T_i
\leq &
\frac{4\gamma_2 K' }{\mu (a) } 
\end{align*}
and we conclude using \eqref{gammadeux} and the argument
\eqref{contradiction}.

\subsection{$\boxed{\alpha' = \alpha>1, d'<d, \text{ and }\beta'=\beta >1$, $e'<e}$}

From Proposition \ref{gap-prefix}, it comes that $d-d'= e-e'$.

\subsubsection{$\alpha $ and $\beta$ are multiplicatively independent}

From \eqref{ineg-T1} and \eqref{nu} we have

\begin{align*}
T_i
\leq &
\frac{4\gamma_2 K'}{\mu (a)} 
\frac{p_i^{e'-e}}{q_i^{d'-d}} 
= 
\frac{4\gamma_2 K'}{\mu (a)} 
\left(\frac{p_i}{q_i}\right)^{e'-e} 
\longrightarrow \frac{4\gamma_2 K'}{\mu (a)} 
\left(\frac{\log \alpha}{\log \beta}\right)^{e'-e} <1 .
\end{align*}

Using  \eqref{gammadeux} we obtain  $T_i$ is strictly smaller than $1$ for some large enough $i$.
We conclude with the argument \eqref{contradiction}.

\subsubsection{$\alpha $ and $\beta$ are multiplicatively dependent}
\label{case-lin-dep}

We can suppose $\alpha = \beta $.
From the hypothesis, we necessarily have $d\not = e$.
From \eqref{ineg-T1} and for $i\geq j_0$ we have:

\begin{align*}
T_i
\leq 
\frac{4\gamma_2 K'}{\mu (a)} 
\left(\frac{p_i}{q_i}\right)^{e'-e} .
\end{align*}

From \eqref{nu} we observe that $\lim_{i\to \infty }p_i/q_i = 1$. 
We conclude using \eqref{gammadeux}.

\subsection{$\boxed{\alpha > \alpha' >1\text{ and }\beta > \beta' >1}$}

From Proposition \ref{gap-prefix}, we necessarily have 
$$\frac{\log \alpha'}{\log \alpha}= \frac{\log \beta'}{\log \beta}\quad\text{ and }\quad
e' -e\frac{\log \alpha'}{\log \alpha} = d' -d\frac{\log \alpha'}{\log \alpha}.$$

\subsubsection{$\alpha$ and $\beta$ multiplicatively independent}

From \eqref{ineg-T3} and \eqref{nu}  we have:

$$
T_i
\leq 
\frac{2K' (2\gamma_2)^{\frac{\log \alpha'}{\log \alpha}}}{\mu (a)}
\cdot
\left(
\frac{p_i}{q_i}
\right)^{e' -e\frac{\log \alpha'}{\log \alpha}}
\longrightarrow
\frac{2K' (2\gamma_2)^{\frac{\log \alpha'}{\log \alpha}}}{\mu (a)}
\cdot
\left(
\frac{\log \alpha}{\log \beta}
\right)^{e' -e\frac{\log \alpha'}{\log \alpha}}
$$
which is, from \eqref{gammadeux}, strictly smaller than $1$ for large enough $i$.

\subsubsection{$\alpha $ and $\beta$ are multiplicatively dependent}

We can suppose $\alpha = \beta $.
We necessarily have $d\not = e$.
It suffices to proceed as in the paragraph \ref{case-lin-dep}.

\subsection{$\boxed{\alpha > \alpha' =1\text{ and }\beta > \beta' =1}$}\label{finn}

From Proposition \ref{gap-prefix} we obtain that $e'-d'\leq 1$ and $d'-e'\leq 1$,
hence $|d'-e'|\leq 1$.

\subsubsection{$\alpha $ and $\beta$ are multiplicatively independent}

From \eqref{ineg-S1} and \eqref{ineg-T1} and for $i\geq j_0$ we have:

\begin{align*}
S_i
\leq 
\frac{K'}{\mu (b) } 
\frac{n_i^{d'}}{m_i^{e'}}
\hbox{ and }
T_i
\leq 
\frac{ K'}{\mu (a)} 
\frac{p_i^{e'}}{q_i^{d'}} .
\end{align*}

\medskip

a) Suppose $|e'-d'| = 1$.  From \eqref{nu} we deduce that either
$(T_i)_{i\in\NN}$ or $(S_i)_{i\in\NN}$ tends to $0$ for $i$ tending to
infinity.

b) Suppose $e'=d'$.
In this case for $i$ sufficiently large we have 

$$
T_i
\leq 
\frac{ K'}{\mu (a)} 
\left(\frac{p_i}{q_i} \right)^{e'}
\leq
\frac{2 K'}{\mu (a)} \left(\frac{\log \alpha}{\log \beta } \right)^{e'} .
$$

We conclude using \eqref{gammadeux}.

\subsubsection{$\alpha$ and $\beta$ multiplicatively dependent}

W.l.o.g. we suppose $\alpha = \beta$.
We necessarily have $\alpha = \beta >1$ and $d\not = e$.
From Proposition \ref{gap-prefix}, we obtain $|d'-e'|\leq 1$.

\medskip

a) Suppose $e'=d'$.
From \eqref{ineg-T1} and for $i\geq j_0$ we have:
\begin{align*}
T_i
\leq 
\frac{2 K'}{\mu (a)} 
\left(\frac{p_i}{q_i}\right)^{e'} .
\end{align*}

But, from \eqref{nu} we know $(p_i/q_i)_i$ tends to $1$.
We conclude using \eqref{gammadeux}.

\medskip

b) $d'=e'+1$.
From \eqref{ineg-T1} and for $i\geq j_0$ we have:

\begin{align*}
T_i
\leq 
\frac{2 K'}{\mu (a)} 
\frac{p_i^{e'}}{q_i^{e'+1}} .
\end{align*}

Using \eqref{nu} $(T_i)$ clearly goes to $0$.

\medskip

c) $e'=d'+1$. 
It can be treated as the case b).

\subsection{Consequence for the words and application to abstract numeration systems}

In the previous section we proved under the assumptions of Theorem
\ref{the:boundedgap} that the letters having infinitely many
occurrences in $x$ appear in $x$ with bounded gaps.  In this section
we deduce that the same result holds not only for letters but also for
words.

Consequently, we obtain an analogue of Cobham's theorem for one
substitution of polynomial growth (the other being exponential).
Theorem \ref{cobb} combined with the main theorem of \cite{Du1} leads
therefore to an extended version of Cobham's theorem. This latter
result expressed in terms of subsitutions can easily be translated
into the formalism of abstract numeration systems (see Corollary
\ref{abs1} and Remark \ref{abs2}).

\begin{coro}
\label{cor:boundgap}
Under the assumptions of Theorem \ref{the:boundedgap}, the words having
infinitely many occurrences in $x$ appear in $x$ with bounded gaps.
\end{coro}

\begin{proof}
  The proof is essentially the same as in \cite{Du1}. Let $u$ be a word
  having infinitely many occurrences in $x$. We set $|u|=n$. To prove
  that $u$ appears with bounded gaps in $x$ it suffices to prove that
  the letter $1$ appears with bounded gaps in the sequence
  $t\in\{0,1\}^\mathbb{N}$ defined by
  $$t_i=1,\quad\text{ if }\ x_{[i,i+n-1]}=u;$$
  and $0$ otherwise. 
  
  The sequence $y^{(n)} = ((y_i \cdots y_{i+n-1}) ; i\in \NN )$ is a
  fixed point of the substitution $\sigma_n : A_n \rightarrow A_n^*$
  where $A_n$ is the alphabet $A^n $, defined for all $(a_1\cdots
  a_n)$ in $A_n$ by
$$
\sigma_n ((a_1\cdots a_n)) = (b_1\cdots b_n)(b_2\cdots
b_{n+1})\cdots (b_{|\sigma (a_1)|}\cdots b_{|\sigma (a_1)|+n-1})
$$
where $\sigma(a_1\cdots a_n) = b_1\cdots b_k$ (for more details see
Section V.4 in \cite{Qu} for example).

Let $\rho : A_n \rightarrow A^{*}$ be the letter-to-letter morphism
defined by $\rho ( ( b_1\cdots b_n )) = b_1$ for all $(b_1\cdots
b_n)\in A_n$. We have $\rho \circ \sigma_n = \sigma \circ \rho$, and
then $\rho \circ \sigma_n^k = \sigma^k \circ \rho$.  $\sigma$ is of
growth type $(\alpha,d)$ then $y^{(n)}$ is $(\alpha,d)$-substitutive.

Let $f : A_n \rightarrow \{ 0,1 \}$ be the letter-to-letter morphism
 defined by

$$
f ( ( b_1\cdots b_n )) = 1 \hbox{ if } b_1\cdots b_n = u \hbox{ and } 0 \hbox{ otherwise.}
$$

It is easy to see that $f(y^{(n)})=t$ hence $t$ is
$(\alpha,d)$-substitutive. 
We proceed in the same way with $\tau$ and Theorem \ref{the:boundedgap}
concludes the proof.
\end{proof}

\begin{lemma}\cite[Th\'eor\`eme 4.1]{Pa}
\label{pan}
  Let $x$ be a proper fixed point of a substitution $\sigma : A\to
  A^{*}$. Let $B$ be the set of non-growing letters of $A$. If in $x$
  occur arbitrarily long words belonging to $B^*$, then there exists
  a growing letter $a\in A$ and $i\in \NN$ such that $\sigma^i (a) =
  vau$ (or $uav$) with $u\in B\setminus \{ \epsilon \}$.
\end{lemma}

\begin{theo}\label{cobb}
  Let $x\in C^\NN$ being both $(d,\alpha)$-substitutive and
  $(e,\beta)$-substitutive for two substitutions satisfying the point
  (3) of the hypothesis of Theorem \ref{the:boundedgap}.  Then $x$ is
  ultimately periodic.
\end{theo}

\begin{proof}
  From Theorem \ref{cor:boundgap} we know that the words appearing
  infinitely many times in $x$ occur with bounded gaps in $x$.
  Suppose $\beta = 1$ and let $z$ be the fixed point of $\tau$ that
  projects onto $x$.  The substitution $\tau$ being polynomial one can
  prove that there exists a word $u$ for which $u^n$ occurs in $z$ for
  all $n$ (for the sake of completeness, we recall Lemma \ref{pan}).
  We assume there is no shorter word having this property.  Then using
  the arguments of Theorem 18 in \cite{Du1} we achieve the proof.
\end{proof}

\begin{coro}\label{abs1}
  Let $S=(L,\Sigma,<)$ (resp. $T=(M,\Gamma,\prec)$) be an abstract
  numeration system where $L$ is a polynomial regular language (resp.
  $M$ is an exponential regular language). If a set $X$ of integers is
  both $S$-recognizable and $T$-recognizable, then $X$ is a finite
  union of arithmetic progressions.
\end{coro}

\begin{proof}
  This is a direct consequence of Theorem \ref{cobb} and the
  discussion made in subsection \ref{sublink}.
\end{proof}

\begin{rema}\label{abs2}
  If $S=(L,\Sigma,<)$ and $T=(M,\Gamma,\prec)$ are abstract numeration
  systems built on two exponential languages then a Cobham's theorem
  holds with the same assumptions as the ones considered in
  \cite{Du1}.
\end{rema}

We address the following conjecture for which partial answers are
given here and in \cite{Du1}.

\begin{conj}
Let $\sigma$ and $\tau$ be two independent substitutions having proper fixed points mapped on the sequence $x$
by letter-to-letter morphisms.
Then $x$ is ultimately periodic.
\end{conj}


\end{document}